\documentclass[a4paper]{amsart}

\usepackage{latexsym, amssymb,amsmath, amsthm,amsfonts}

\newcommand{\kk}{\Bbbk}

\def\SL{\operatorname{SL}}

\def\SL2{\operatorname{SL}_{2}(K)}

\def\GL2{\operatorname{GL}_{2}(K)}

\def\INVSL2{$K[V]^{operatorname{SL}_{2}(K)}$}
\def\INVSO2{$K[V]^{operatorname{SO}_{2}(K)}$}
\def\INVGL2{$K[V]^{operatorname{GL}_{2}(K)}$}

\def\GL{\operatorname{GL}}
\def\SL{\operatorname{SL}}

\def\Quot{\operatorname{Quot}}

\def\codim{\operatorname{codim}}

\def\wt{\operatorname{wt}}
\def\Tr{\operatorname{Tr}}

\newtheorem{Lemma}{Lemma}
\newtheorem{Theorem}[Lemma]{Theorem}

\newtheorem{Corollary}[Lemma]{Corollary}

\newtheorem{prop}[Lemma]{Proposition}

\newtheorem*{Corollary of Conjecture}{Corollary of Conjecture}

\theoremstyle{definition}

\theoremstyle{remark}

\newtheoremstyle{Acknowledgments}
  {}
    {}
     {}
     {}
    {\bfseries}
    {}
     {.5em}
     {\thmname{#1}\thmnumber{ }\thmnote{ (#3)}}
\theoremstyle{Acknowledgments}
\newtheorem{ack}{Acknowledgments.}

\title[Cohen-Macaulay covariants]{Cohen-Macaulay modules of covariants for cyclic-$p$-groups}

\author{Jonathan Elmer}
\address{Middlesex University\\
The Burroughs, Hendon, London\\
NW4 4BT UK}
\email{j.elmer@mdx.ac.uk}

\date{\today}
\subjclass[2010]{13A50}
\keywords{Invariant theory,  module of covariants, free module, Cohen-Macaulay module}

\begin{document}

\begin{abstract}

Let $G$ be a a finite group, $\kk$ a field of characteristic dividing $|G|$ and and $V,W$ $\kk G$-modules. Broer and Chuai \cite{BroerChuaiRelative} showed that if $\codim(V^G) \leq 2$ then the module of covariants $\kk[V,W]^G = (\kk[V] \otimes W)^G$ is a Cohen-Macaulay module, hence free over a homogeneous system of parameters for the invariant ring $\kk[V]^G$.

In the present article we prove a general result which allows us to determine whether a set of elements of a free $A$-module $M$ is a generating set, for any $\kk$-algebra $A$. We use this result to find generating sets for all modules of covariants $\kk[V,W]^G$ over a homogeneous system of parameters, where $\codim(V^G) \leq 2$ and $G$ is a cyclic $p$-group.
\end{abstract}

\maketitle

\section{Introduction}  

Let $G$ be a finite group, $\kk$ a field and let $V,W$ be a pair of finite-dimensional $\kk G$-modules. Then $G$-acts on the set $\kk[V,W] = S(V^*) \otimes W$ of polynomial morphisms $V \rightarrow W$ via the formula
\[(g \phi )(v) = g \phi (g^{-1} v). \]
Morphisms fixed under the action of $G$ are called \emph{covariants}, and we denote the set of covariants by $\kk[V,W]^G$. In the special case $W =\kk$ we write $\kk[V,\kk] = \kk[V]$ and the corresponding fixed points $\kk[V]^G$ are called invariants. The set of invariants is a $\kk$-algebra, and pointwise multiplication endows $\kk[V,W]^G$ with the structure of a $\kk[V]^G$-module.

Many theorems about invariant algebras have analogues for covariants which are less well-known. For example, suppose $\kk$ has characteristic zero. Then it is well-known that $\kk[V]^G$ is a polynomial ring if and only if $G$ is generated by reflections (elements fixing a subspace of codimension 1 in $V$). It is less well-known that these conditions are equivalent to $\kk[V,W]^G$ being a free module over $\kk[V]^G$ for all $W$ \cite{Chevalley}, \cite{SheppardTodd}. Similarly, it is well-known that $\kk[V]^G$ is always a Cohen-Macaulay algebra. It is less well-known that $\kk[V,W]^G$ is always a Cohen-Macaulay module over $\kk[V]^G$ \cite{HochsterEagon}.

In the modular case, much less is known about the structure of covariant modules. The following result is taken from \cite{BroerChuaiRelative}:

\begin{prop}\label{broerchuaimain}
Let $\kk$ be field whose characteristic divides $|G|$, where $G$  is a finite group acting on vector spaces $V,W$. 

\begin{itemize}
\item Suppose that $\codim(V^G) = 1$. Then $\kk[V,W]^G$ is a free $\kk[V]^G$-module.

\item Suppose that $\codim(V^G) \leq 2$. Then $\kk[V,W]^G$ is a Cohen-Macaulay $\kk[V]^G$-module.
\end{itemize}
\end{prop}

In addition, Broer and Chuai give necessary and sufficient conditions for a set of covariants to generate $\kk[V,W]^G$ freely over $\kk[V]^G$. 

Suppose $\kk[V,W]^G$ is a Cohen-Macaulay $\kk[V]^G$-module. This means that there exists a polynomial subalgebra $A \subseteq \kk[V]^G$ (generated by a homogeneous system of parameters) over which $\kk[V,W]^G$ is finitely generated and free. 

The purpose of the present article is to give a method to test whether a set of covariants generates $\kk[V,W]^G$ over $A$.. We obtain explicit sets of covariants generating $\kk[V,W]^G$ freely over a homogeneous system of parameters when $G$ is a cyclic $p$-group, $\codim(V^G) \leq 2$ and $W$ is arbitrary. Note that in this case we have that $\kk[V]^G$ itself is Cohen-Macaulay (apply Proposition \ref{broerchuaimain} with $W$ trivial). In fact, it follows from \cite{Ellingsrud} and \cite{FleischmannCohomConnectivity} that $\kk[V]^G$ is Cohen-Macaulay \emph{if and only if} $\codim(V^G) \leq 2$.

Suppose $G$ is a cyclic group of order $q = p^k$. Recall the following concerning representation theory of $G$; there are exactly $q$ isomorphism classes of indecomposable modules for $G$. The dimensions of these are $1,2,3, \ldots, q$, and we denote by $V_r$ the unique indecomposable of dimension $r$. A generator $\sigma$ of $G$ acts on $V_r$ by left-multiplication with a Jordan block of size $r$ with eigenvalue 1. If $r \leq p^l$ for $l<k$ we have that $\sigma^{p^l}$ acts trivially on $V_r$; thus, $V_r$ is faithful only when $r>p^{k-1}$. The dimension of the fixed-point space $V^G$ is one for any indecomposable $V$. As a consequence, we obtain:

\begin{prop}\label{listmodules} Suppose $V$ is a faithful representation of a cyclic group $G$ of order $p^k$, and that $\codim(V^G) \leq 2$. Suppose in addition that $V$ contains no trivial direct summands. Then one of the following hold:

\begin{enumerate}
\item $|G| = p$ and $V \cong V_2$;

\item $|G| = p>2$ and $V \cong V_3$;

\item $|G| = p$ and $V \cong V_2 \oplus V_2$.

\item $|G| = 4$ and $V \cong V_3$;

\end{enumerate}
\end{prop}
 
This paper is organised as follows: in the next section we will prove a very general result which allows us to test whether a set of elements of a free module is a basis (Proposition \ref{freetest}). In the third section we will show how modular covariants for cyclic groups can be viewed as the kernel of a certain homomorphism. This generalises \cite[Proposition~4]{ElmerCov}. As an incidental result we show that the kernel of the relative transfer map $\Tr^G_H: \kk[V]^H \rightarrow \kk[V]^G$ can also be viewed as a module of covariants. 

The last section gives explicit generating sets of covariants over a homogeneous system of parameters for the four cases explained above. In the first two cases, such generating sets were obtained in \cite{ElmerCov}, but the proof here using Proposition \ref{freetest} is shorter in the second case, as we can circumvent computation of the Hilbert series of $\kk[V,W]^G$. The results in the third and fourth cases are new.

\begin{ack}
The author would like to thank an anonymous referee of \cite{ElmerCov} for the communication which inspired section 3 of this paper.
\end{ack}

\section{The $s$-invariant and freeness}

Let $\kk$ be a field and $A$ a $\kk$-algebra. Let $M$ be a finitely-generated module over $A$. A set $g_1,g_2, \ldots, g_r$ of elements of $M$ is called a {\it generating set} if $M = A(g_1,g_2,\ldots, g_r)$. It is said to be $A$-\emph{independent} if we have
\[\sum_{i=1}^r a_ig_i  = 0 \Rightarrow a_1=a_2= \ldots = a_r=0 \ \text{for all} \ a_1,a_2, \ldots, a_r \in A.\] An $A$-independent, generating set for $M $ is called a \emph{basis} of $M$. 

The most familiar setting is where $A = \kk$. In that case $M$ is a finite-dimensional vector space. It is then well known that $M$ has a basis, and any maximal $A$-independent set, or minimal generating set for $M$ is a basis. This may not be the case for modules over other algebras. A module which has a basis is called \emph{free}, and one can show that every free module over $A$ is isomorphic to $A^r$ for some $r$.

Now suppose $A = \oplus_{i \geq 0} A_i$ is a graded $\kk$-algebra with $A_0 = \kk$. Let $M = \oplus_{i \geq 0} M_i$ be a graded $A$ module. We have Hilbert series

\[H(A,t) = \sum_{i \geq 0} \dim(A_i)t^i,\]
\[H(M,t) = \sum_{i \geq 0} \dim(M_i)t^i.\]

Consider the quotient of these series, expanded about $t = 1$:

\[H(M, A; t):= \frac{H(M,t)}{H(A,t)} = a_0 + a_1(t-1) + \ldots  \]

If $M$ is finitely generated, then by the Hilbert-Poincar\'e theorem we can write $$H(M,t) = \frac{f(t)}{H(A,t)}$$ for some polynomial $f(t)$. 
In that case we have 
\[a_0 = f(1) = r(M,A).\]
This is called the \emph{rank} of $M$ over $A$. If $M$ has finite projective dimension (for example, if $A$ is regular) it can be shown to equal the dimension of $M$ as a vector space over $\Quot(A)$. We also have
\[a_1= f'(1) = s(M,A). \] This is called the \emph{$s$-invariant} of $M$ over $A$.

The following is an easy consequence of the definition of $H(M, A;t)$:

\begin{prop}[\cite{BroerChuaiRelative}]\label{ssubalg} Let $A$ be a $\kk$-algebra and $B \subset A$ a subalgebra of $A$ over which $A$ is finitely generated. Then we have
\begin{itemize}
\item[(a)] \[r(M,B) = r(M,A)r(A,B);\]
\item[(b)] \[s(M,B) = r(M,A)s(A,B) + r(A,B)s(M,A).\]
\end{itemize}
\end{prop}

The next result is the main result of this section, and will be used frequently in the remainder of the article.

\begin{prop}\label{freetest} Let $A$ be a regular, graded $\kk$-module with $A_0= \kk$ and let $M$ be a finite generated free graded module over $A$. Let $g_1, g_2, \ldots, g_r$ be a $A$-independent set of elements of $M$, where $r = r(M,A)$. Then
\[\sum_{i=1}^r \deg(g_i) \geq s(M,A)\] with equality if and only if $M$ is generated by $g_1, \ldots, g_r$.
\end{prop}

\begin{proof} Let $f_1, f_2, \ldots, f_r$ be a generating set of $M$.  We may assume without loss of generality that 
\[\deg(f_1) \leq \deg(f_2) \leq \ldots \leq \deg(f_r)\] and also
\[\deg(g_1) \leq \deg(g_2) \leq \ldots \leq \deg(g_r).\] As $M$ is free, we have
\[\frac{H(M,t)}{H(A,t)} = \sum_{j=1}^r t^{\deg(f_j)}\] and $s(M,A) = \sum_{j=1}^r \deg(f_j) $.

We claim that $\deg(g_j) \geq \deg(f_j)$ for all $j \leq r$. Suppose the contrary. Then for some $k \leq r-1$ we must have $\deg(g_{k+1})< \deg(f_{k+1})$ and consequently
\[g_j \in A(f_1, \ldots, f_k)\] for all $j = 1, \ldots, k+1$. 
Write $g_j = \sum_{i=1}^k a_{ij} f_i$. Let $A$ be the $k \times( k+1)$ matrix with elements \[ (a_{ij}: i=1, \ldots, k, j = 1, \ldots, k+1.)\]
We will prove by reverse induction on $t$ that all $t \times t$ minors of $A$ are zero. The initial case is $t=k$. Form the matrix $A'$ by adding an extra row $\{g_1, \ldots, g_{k+1}\}$. Now the matrix $A'$ is

\[\begin{pmatrix} g_1 & g_2 & \ldots & g_{k+1} \\ a_{11} & a_{12} & \ldots & a_{1 k+1} \\
\vdots & \vdots & & \vdots \\ a_{k1} & a_{k2} & \ldots & a_{k k+1} \end{pmatrix}\] 

The determinant of this matrix is a $A$-linear combination of $f_1, \ldots, f_{k}$, and the coefficent of $f_i$ is

\[\det \begin{pmatrix} a_{i1} & a_{i2} & \ldots & a_{i k+1} \\ a_{11} & a_{12} & \ldots & a_{1 k+1} \\
\vdots & \vdots & & \vdots \\ a_{k1} & a_{k2} & \ldots & a_{k k+1} \end{pmatrix}\] which is zero because it has a repeated row. We thus obtain a relation
\[\sum_{j=1}^{k+1} (-1)^j g_j A_j = 0\] where $A_j$ is the $k \times k$ minor of $A$ obtained by deleting the $j$th column. Since $g_1, g_2, \ldots, g_{k+1}$ are linearly independent over $A$, we get that $A_j = 0$ as required.

Now suppose that all $(t+1) \times (t+1)$ minors of $A$ are zero, where $t<k$. Consider a $t \times t$ submatrix
\[B:= (a_{p_i q_j}: i,j = 1, \ldots, t).\] Pick any $q_{t+1} \in \{1, \ldots, k+1\} \setminus \{q_1, q_2, \ldots, q_t\}$ and consider the determinant of the matrix 

\[B':= \begin{pmatrix} g_{q_1} & g_{q_2} & \ldots & g_{q_{t+1}} \\ a_{p_1q_1} & a_{p_1q_2} & \ldots & a_{p_1 q_{t+1}} \\
\vdots & \vdots & & \vdots \\ a_{p_t q_1} & a_{p_t q_2} & \ldots & a_{p_t q_{t+1}} \end{pmatrix}\] 

This is again an $A$-linear combination of $f_1, \ldots, f_k$, and the coefficient of $f_i$ is

\[\det \begin{pmatrix} a_{i q_1} & a_{i q_2} & \ldots & a_{i q_{t+1}} \\ a_{p_1q_1} & a_{p_1q_2} & \ldots & a_{p_1 q_{t+1}} \\
\vdots & \vdots & & \vdots \\ a_{p_t q_1} & a_{p_t q_2} & \ldots & a_{p_t q_{t+1}}. \end{pmatrix}\] 
If $i \not \in \{p_1, \ldots, p_t\}$ then this is a $t+1 \times t+1$ minor of $A$, and zero by inductive hypothesis. Otherwise it is again zero, being the determinant of a matrix with a repeated row. We thus obtain
\[0 = \sum_{j=1}^{t+1} (-1)^j g_{q_j} B_{j}\] where $B_{j}$ is the $t \times t$ minor of $B'$ obtained by removing the $j$th column. Since $g_{q_1}, \ldots, g_{q_t}$ are linearly independent we obtain $B_{j} = 0$ for all $j=1, \ldots, t$, in particular, \[\det(B) = B_t = 0.\]

This completes the induction. Taking $t=1$, we see that $a_{ij} =0$ for all $i = 1, \ldots, k$ and $j = 1, \ldots, k+1$. Therefore 
\[g_1 = g_2 = \ldots g_{k+1} = 0\] which contradicts linear independence.

This completes the proof of our claim, and we may conclude that $\sum_{j=1}^r \deg(g_j) \geq \sum_{j=1}^r \deg(f_j).$
Further, if $\sum_{j=1}^r \deg(g_j) = \sum_{j=1}^r \deg(f_j)$ then we must have $\deg(g_j) = \deg(f_j)$ for all $j=1, \ldots, r$, and $A(g_1, g_2, \ldots, g_r)$ is a submodule of $M$ with the same Hilbert series, and therefore $g_1, g_2, \ldots, g_r$ are free generators of $M$. On the other hand, if $\sum_{j=1}^r \deg(g_j) > \sum_{j=1}^r \deg(f_j)$ then $\deg(g_j)> \deg(f_j)$ for some $j $, and $A(g_1, g_2, \ldots, g_r)$ is a proper submodule of $M$.
 \end{proof}

\section{Characterising covariants}

In this section we will show how modules of covariants can often be viewed as the kernel of a certain homomorphism. First, let $G$ denote a finite group and $\kk$ a field of arbitary characteristic. Let $H \leq G$. Suppose $G$ acts on a $\kk$-vector space $V$ and let $T$ be a left-transveral of $H$ in $G$. We can assume $\iota \in T$ where $\iota$ is the identity element in $G$. Let $\{e_t: t \in T\}$ be a basis for the left permutation module $\kk(G/H)$ on which $H$ acts trivially. We define a map

\[\Theta: \kk[V]^H \rightarrow \kk[V] \otimes \kk(G/H)\]
\[\Theta(f) = \sum_{t \in T} tf \otimes e_t.\]
where $G$ acts diagonally on the tensor product.

Clearly $\Theta$ is an injective, degree-preserving homomorphsim of  $\kk[V]^G$ modules. Further we claim:
\begin{prop} $\Theta$ induces an isomorphism  of graded $\kk[V]^G$-modules $$\kk[V]^H \cong (\kk[V] \otimes \kk(G/H))^G.$$
\end{prop}

\begin{proof} We show first that if $f \in \kk[V]^H$ then $\Theta(f) \in (\kk[V] \otimes \kk(G/H))^G$. Let $g \in G$, then we have

$$g \cdot \Theta(f) = \sum_{t \in T} gt f \otimes ge_t.$$

Write $gt = \sigma_g(t) h_g$, where $\sigma_g(t) \in T$ and $h_g \in H$, noting that $\sigma_g$ is a permutation of $T$. Then since $f \in \kk[V]^H$ and $H$ acts trivially on $\kk(G/H)$ we have
\[ \sum_{t \in T} gt f \otimes ge_t = \sum_{t \in T} \sigma_g(t) h_g f \otimes h_g e_{\sigma_g(t)} = \sum_{t \in T} \sigma_g(t) f \otimes e_{\sigma_g(t)} = \Theta(f)\] as required

 Now we need only show that $\Theta$ is surjective. Let $$f = \sum_{t \in T}^r f_t \otimes e_t \in (\kk[V]\otimes \kk(G/H))^G$$ where $f_t \in \kk[V]$. Again for any $g \in G$ write $gt = \sigma_g(t)h_g.$ Thus since $f = gf$ we have
$$f = \sum_{t \in T} gf_t \otimes e_{\sigma_g(t)}= \sum_{t \in T} f_t \otimes e_t.$$
Equating coefficients of $e_t$ shows that
\[gf_t = f_{\sigma_g(t)}\] for any $g \in G$ and $t \in T$. In particular, if $t \in T$ then the permutation $\sigma_t$ satisfies $\sigma_t(\iota) = t$. Thus for any $t \in T$ we have
\[t f_{\iota} = f_t\] and hence 
\[f =  \sum_{t \in T}^r f_t \otimes e_t = \sum_{t \in T} tf_{\iota} \otimes e_t = \Theta(f_{\iota})\] which shows that $\Theta$ is surjective as required.
\end{proof}

Now let $M$ be a submodule of $\kk(G/H)$. Tensor the short exact sequence
\[0 \rightarrow M \stackrel{i}{\rightarrow} \kk(G/H) \stackrel{j}{\rightarrow} N \rightarrow 0, \] where $N = \ker(i)$ with $\kk[V]$ and take $G$-invariants. This induces a long exact sequence
\[0 \rightarrow (\kk[V] \otimes M)^G \rightarrow (\kk[V] \otimes \kk(G/H))^G \rightarrow (\kk[V] \otimes N)^G \rightarrow H^1(G,\kk[V] \otimes M) \rightarrow \ldots \]

By composing the isomorphism $\Theta$ with the second nontrivial map in this sequence, we obtain a graded $\kk[V]^G$ homomorphism 
$$\phi: \kk[V]^H \rightarrow (\kk[V] \otimes N)^G$$
whose kernel is isomorphic to the module of covariants $(\kk[V] \otimes M)^G = \kk[V,M]^G$.

We will consider two special cases of the above: first, suppose $M = M_{G/H} = \langle e_t - e_\iota: t \in T\rangle_{\kk}$. Then $N$ is a one-dimensional module on which $G$ acts trivially and the map $j$ sends each $e_t$ to its basis element, which we will denote by $e$. The induced map in the long exact sequence sends an expression $$\sum_{t \in T} f_t \otimes e_t$$ to $$\sum_{t \in T} f_t \otimes e.$$ Thus, the composition with $\Theta$ sends an element $f \in \kk[V]^H$ to $\sum_{t \in T} tf$ - in other words, this is the \emph{relative transfer map} $\Tr^G_H: \kk[V]^H \rightarrow \kk[V]^G$. It follows that the kernel of the relative transfer map is isomorphic, as a graded $\kk[V]^G$-module, to the module of covariants $\kk[V,M_{G/H}]^G$. This is interesting, but we will not pursue it further in the present article. We note for future use that the (relative) transfer map is a useful way of producing invariants. Along similar lines, one has a \emph{relative norm map} $N^G_H:\kk[V]^H \rightarrow \kk[V]^G$ defined by
\[N^G_H(f) = \prod_{t \in T} tf.\]

Now suppose $G$ is cyclic of order $q = p^r$, $\kk$ is a field of characteristic $p$ and $H$ is trivial. Let $g$ be a generator of $G$ and set $\Delta:= g-\iota \in \kk G$. For each $1 \leq n < q$ there is a submodule $\ker(\Delta^n)$ of $\kk G$ isomorphic to $V_n$.  The image of $\Delta^n$ is a submodule isomorphic to $V_{q-n}$. 

Let $M$ and $N$ denote the kernel and image of $\Delta^n$ respectively. Then the canonical short exact sequence induces a long exact sequence

\[0 \rightarrow (\kk[V] \otimes M)^G \rightarrow (\kk[V] \otimes \kk(G))^G \stackrel{\psi}{\rightarrow} (\kk[V] \otimes N)^G \rightarrow H^1(G,\kk[V] \otimes M) \rightarrow \ldots \]

Consider the composition 
\[\kk[V] \stackrel{\Theta}{\rightarrow} (\kk[V] \otimes \kk G)^G \stackrel{\psi}{\rightarrow} (\kk[V] \otimes N)^G \stackrel{k} {\rightarrow} (\kk[V] \otimes \kk G)^G  \stackrel{\Theta^{-1}}{\rightarrow } \kk[V]\]
Here $k$ is the canonical inclusion.  Choose a basis $\{e_i: i=0, \ldots q-1\}$ for $\kk G$ with $ge_i = e_{i+1}$ for $i<q$ and $ge_{q-1} = e_0$.
For $n=1$, the composition $k \psi$ sends an element 
\[\sum_{i=0}^{q-1} f_i \otimes e_i \] to
\[\sum_{i=0}^{q-1} f_i \otimes (e_{i+1}-e_i) = \sum_{i=0}^{q-1} (f_{i-1}-f_{i}) \otimes e_i \] with indices understood modulo $q$. Thus the composition $k \psi \Theta$ sends $f \in \kk[V]$ to 
\[\sum_{i=0}^{q-1} (g^{i-1} f - g^{i} f) \otimes e_i.\] Composing this with $\Theta^{-1}$ then picks out the coefficient of $e_0$, which is 
\[g^{-1} f  - f.\] Thus, the composition is action by $g^{-1}- \iota \in \kk G$. An easy inductive argument shows that for arbitrary $n$, the composition is action by $(g^{-1} - \iota)^n \in \kk G$. Since $\Theta$ is an isomorphism and $k$ is injective, the kernel of this map (which is the same as the kernel of $\Delta^n$) can be identified with the kernel of $\psi$, which in turn is isomorphic to $\kk[V,M]^G$. Thus, we obtain a degree-preserving isomorphism of $\kk[V]^G$ modules

\[\Xi:\ker(\Delta^n) \cong \kk[V,V_n]^G.\]

Note that a version of this isomorphism was used in \cite{ElmerCov} for $G$ a cyclic group of prime order. As noted there, if we choose a basis $\{w_1,w-2, \ldots, w_n\}$ of $V_n$ such that
\begin{align*} \sigma w_1 &=w_1\\
\sigma w_2&=w_2-w_{1}\\
\sigma w_3 &= w_2-w_2+w_1\\
\vdots & \\
\sigma w_n &=w_n-w_{n-1}+w_{n-2}-\ldots \pm w_1. \end{align*}
then $\Xi$ is given by the particularly convenient formula
\[\Xi(f) = \sum_{i=0}^n \Delta^i(f)w_i.\]

\section{Main results}

From now on let $G = \langle \sigma \rangle$ be a cyclic group of order $q = p^k$, let $\kk$ be a field of characteristic $p$ and let $V$ and $W$ be finite-dimensional $\kk G$-modules.

The operator $\Delta = \sigma-\iota \in \kk G$ will play a major role in our exposition, so we begin with some general results, following \cite{ElmerCov} quite closely. Notice that, for $\phi \in \kk[V,W]^G$ we have
\[\Delta(\phi) = 0 \Rightarrow \sigma \cdot \phi = \phi\] and thus by induction $\sigma^k \phi = \phi$ for all $k$. So $\Delta(\phi)=0$ if and only if $\phi \in \kk[V,W]^G$. Similarly for $f \in \kk[V]$ we have $\Delta(f)=0$ if and only if $f \in \kk[V]^G$. 

$\Delta$ is a \emph{$\sigma$-twisted derivation} on $\kk[V]$; that is, it satisfies the formula

\begin{equation}\label{derivation} \Delta(fg) = f\Delta(g)+\Delta(f)\sigma(g)\end{equation} for all $f, g \in \kk[V]$.  

Further, using induction and the fact that $\sigma$ and $\Delta$ commute, one can show $\Delta$ satisfies a Leibniz-type rule

\begin{equation}\label{Leibniz}
\Delta^k(fg) = \sum_{i=0}^k \begin{pmatrix} k\\i \end{pmatrix} \Delta^i( f) \sigma^{k-i}( \Delta^{k-i} (g)).
\end{equation}

A further result, which can be deduced from the above and proved by induction is the rule for differentiating powers:

\begin{equation}\label{powerrule}
\Delta(f^k) = \Delta(f)\left( \sum_{i=0}^{k-1} f^i \sigma(f)^{k-1-i} \right)
\end{equation}
for any $k \geq 1$.

For any $f \in \kk[V]$ we define the \textbf{weight} of $f$:
\[\wt(f) = \min\{i>0: \Delta^i(f) = 0\}.\]
Notice that $\Delta^{\wt(f)-1}(f) \in \ker(\Delta) = \kk[V]^G$ for all $f \in \kk[V]$. 
Another consequence of  (\ref{Leibniz}) is the following: let $f, g \in \kk[V]$ and set $d = \wt(f), e = \wt(g)$. Suppose that
\[d+e-1 \leq p.\] Then
\[\Delta^{d+e-1}(fg) = \sum_{i=0}^{d+e-1}\begin{pmatrix} d+e-1\\i \end{pmatrix}\Delta^i(f)\sigma^{d+e-1-i}(\Delta^{d+e-1-i}(g)) = 0\]
since if $i<e$ then $d+e-1-i>d-1$. On the other hand
\begin{align*}
\Delta^{d+e-2}(fg) &= \sum_{i=0}^{d+e-2}\begin{pmatrix} d+e-2\\i \end{pmatrix}\Delta^i(f)\sigma^{d+e-2-i}(\Delta^{d+e-2-i}(g)) \\
 &= \begin{pmatrix} d+e-2\\i \end{pmatrix}\Delta^{d-1}(f)\sigma^{e-1}(\Delta^{e-1}(g)) \neq 0 \end{align*}
since $\begin{pmatrix} d+e-2\\i \end{pmatrix} \neq 0 \mod p$. We obtain the followng:

\begin{prop}\label{weightproduct} Let $f, g \in \kk[V]$ with $\wt(f)+\wt(g)-1 \leq p$. Then $\wt(fg) = \wt(f)+\wt(g)-1$.
\end{prop}

Also note that
\[\Delta^q = \sigma^q-1=0\] which shows that $\wt(f) \leq q$ for all $f \in \kk[V]^G$. Finally notice that
\begin{equation}\label{trans} \Delta^{q-1} = \sum_{i=0}^{q-1} \sigma^i \end{equation} which shows that $\Delta^{q-1}(f) = \Tr^G(f)$ for all $f \in \kk[V]$.

Now we assume that $\codim(V^G) \leq 2$, so $V$ is isomorphic to one of the modules listed in Proposition \ref{listmodules}. Our goal is to find, for arbitrary indecomposable $W \cong V_n$, an explicit set of covariants generating $\kk[V,W]^G$ freely over $A$, where $A$ is a homogeneous system of parameters for $\kk[V]^G$. Our strategy is to find in each case an $A$-independent set of $r(K_n,A)$ elements of $K_n:= \ker(\Delta^n)$ whose degree sum equals $s(K_n,A)$. That such a set generates $K_n$ freely over $A$ follows from Proposition \ref{freetest}, and the results of Section 3 imply that applying $\Xi$ to each element in the set yields a free generating set for $\kk[V,W]^G$ over $A$.

Broer and Chuai studied the $s$-invariant for modules of covariants over $\kk[V]^G$ for arbitrary $G$. In particular they proved:

\begin{prop}\label{nopseudoref} Suppose $G$ is a finite group and let $H$ be the subgroup of $G$ generated by all pseudo-reflections (i.e. elements stabilising a subspace of $V$ with codimension 1). Then $s(\kk[V,W]^G,\kk[V]^G) = s(\kk[V,W]^H,\kk[V]^H)$. In particular if $G$ contains no pseudo-reflections, then $s(\kk[V,W]^G,\kk[V]^G) =0$.
\end{prop}

This makes computation of $(\kk[V,W]^G,A)$ using Proposition \ref{ssubalg} quite easy when $G$ is a cyclic $p$-group containing no pseudo-reflections. This holds when $V = V_3$ and $p>2$ or $V = V_2 \oplus V_2$. The only cases remaining in which the ring of invariants is Cohen-Macaulay are $V=V_2$ and $V = V_3, p=2$. In the former case our methods do not yield any substantial improvement, so we will simply quote the results from \cite{ElmerCov} for the sake of completeness:

\begin{prop} Let $p \geq n$, $V=V_2$ and $K_n = \ker(\Delta^n)$.
$K_n$ is a free $\kk[V^G]$-module, generated by $\{x_1^k: k=0, \ldots, n-1\}$.
\end{prop}

\begin{Corollary}\label{mainv2}
Let $p \geq n$, $W = V_n$ and $V = V_2$. The module of covariants $\kk[V,W]^G$ is generated freely over $\kk[V]^G$ by 
\[\{\Xi(x_1^k): k = 0, \ldots, n-1\}.\]
\end{Corollary}

The case $V = V_3$, $p=2$ will be dealt with in the final subsection.
If $V = V_3$ and $p>2$ or $V = V_2 \oplus V_2$ then we have the following:

\begin{Lemma}\label{rsnoref}
Let $p \geq n$ and let $W = V_n$. Let $A \subseteq \kk[V]^G$ be a subalgebra generated by a homogeneous system of parameters. Then $\kk[V,W]^G$ is a free $A$-module and
\[r(\kk[V,W]^G,A) = r(\kk[V]^G,A)n; \]
\[s(\kk[V,W]^G,A) = s(\kk[V]^G,A)n.\] 
\end{Lemma}

\begin{proof}
The first result follows from Proposition \ref{ssubalg}; the second from the same, Proposition \ref{nopseudoref} and the fact that $V$ contains no pseudo-reflections.
\end{proof}

Note that an application of $\Xi$ now shows that
\[r(K_n,A) = r(\kk[V]^G,A)n; \]
\[s(K_n,A) = s(\kk[V]^G,A)n.\] 

We now consider these cases separately in detail.
\subsection{$V = V_3, p>2$}
In this subsection let $V=V_3$, $W = V_n$ and $p$ an odd prime. Then $G$ is cyclic of order $p$. Choose a basis $v_1,v_2, v_3$ of $V$ so that
\begin{align*}\sigma x_1 &= x_1+x_2\\
\sigma x_2 &= x_2+x_3\\
\sigma x_3 &= x_3 \end{align*}
where $x_1,x_2,x_3$ is the corresponding dual basis.

We begin by describing $\kk[V]^G$. This has been done in several places before, for example \cite{DicksonMadison} and \cite[Theorem~5.8]{Peskin}, but we will follow \cite{ElmerCov}. We use a graded lexicographic order on monomials $\kk[V]$ with $x_1>x_2>x_3$. Here and in the next two subsections, if $f \in \kk[V]$ then the {\it lead term} of $f$ is the term with the largest monomial in our order and the {\it lead monomial} is the corresponding monomial.

 It is easily shown that
\begin{align*} a_1 &:= x_3,\\ a_2 &:= x_2^2-2x_1x_3-x_2x_3,\\ a_3 &:= N^G(x_1)=\prod_{i=0}^{p-1} \sigma^i(x_1) \end{align*}
are invariants, and looking at their lead terms tells us that they form a homogeneous system of parameters for $\kk[V]^G$, with degrees $1,2$ and $p$. 

\begin{prop}\label{gensv3} Let $f \in \kk[V]^G$ be any invariant with lead term $x_2^p$. Let $A = \kk[a_1,a_2,a_3]$. Then $\kk[V]^G$ is a free $A$-module, whose generators are $1$ and $f$.
\end{prop}

\begin{proof} See \cite[Proposition~12]{ElmerCov}
\end{proof}

The obvious choice of invariant with lead term $x_2^p$ is $N^G(x_2)$. 
The following observations are consequences of the generating set above.

\begin{Lemma}\label{obs} 
Let $f \in A$. Then the lead term of $f$ is of the form $x_1^{pi}x_2^{2j}x_3^k$ for some positive integers $i,j,k$.
\end{Lemma}

\begin{Lemma}\label{randsofkv3g} We have $r(\kk[V]^G,A) = 2$ and $s(\kk[V]^G,A) = p$.
\end{Lemma}

\begin{Corollary}
We have $r(\kk[V,W]^G, A) = 2n$ and $s(\kk[V,W]^G,A) =np$.
\end{Corollary}

\begin{proof} This follows immediately from Lemma \ref{randsofkv3g} and Lemma \ref{rsnoref}.
\end{proof}

For the rest of this section, we set $l = \frac12 n$ if $n$ is even, with $l = \frac12(n-1)$ if $n$ is odd. Next, we need some information about the lead monomials of certain polynomials. The following pair of lemmas were proved in \cite{ElmerCov}:

\begin{Lemma}\label{monomialsx_1^ismall} Let $j \leq k < p$. Then $\Delta^j(x_1^k)$ has lead term $$\frac{k!}{(k-j)!}x_1^{k-j}x_2^j.$$
\end{Lemma}
\begin{Lemma}\label{monomialsx_1^ix_2} Let $j \leq k < p$. Then $\Delta^j(x_1^kx_2)$ has lead term $$\frac{k!}{(k-j)!}x_1^{k-j}x_2^{j+1}.$$
\end{Lemma}

We are now ready to state our main results. Let $V=V_3$ and $W=V_n$. For any $i = 0,1, \ldots, n-1$ we define monomials
\[M_i =  \left\{ \begin{array}{lr} x_1^{i/2} & \text{if $i$ is even,}\\ x_1^{(i-1)/2}x_2 & \text{if $i$ is odd}.\end{array} \right.\]
and polynomials 
\[P_i =  \left\{ \begin{array}{lr} \Delta(x_1^{p-i/2}) & \text{if $i$ is even, $i>0$,}\\ x_1^{p-(i+1)/2} & \text{if $i$ is odd}.\end{array} \right.\] with $P_0 = x_1^{p-1}x_2$.

\begin{Theorem}
\label{kngens}
Let $n \leq p$. Then $K_n$ is a free $A$-module, generated by 
\[S_n = \{M_0,M_1, \ldots, M_{n-1}, \Delta^{p-n}(P_0), \Delta^{p-n}(P_1), \ldots, \Delta^{p-n}(P_{n-1})\}.\]
\end{Theorem}

\begin{proof}
By Lemma \ref{weightproduct}, the weight of $M_i$ is $i+1$ for $i<p$, while the weight of $P_i$ is
\[\left\{\begin{array}{lr} p & i\ \text{odd or zero} \\ p-1 & i\ \text{even, $i>0$.} \end{array} \right.\] Therefore the given polynomials all lie in $K_n$. Further, the degree of $M_i$ is $\lceil \frac{i}{2} \rceil$ and the degree of $P_i$ is $p-\lceil \frac{i}{2} \rceil$, so the sum of the degrees of the elements of $K_n$ is $np = s(K_n,A)$. Therefore by Proposition \ref{freetest} it is enough to prove that $S_n$ is $A$-independent.

Applying Lemmas \ref{monomialsx_1^ismall} and \ref{monomialsx_1^ix_2}, the lead monomials of $S_n$ are
\[\{1,x_2, x_1, x_1x_2, \ldots, x_1^{l-1}x_2, x_1^l,\]\[ x_1^{n-l-1}x_2^{p-n+1}, x_1^{n-l}x_2^{p-n}, \ldots, x_1^{n-2}x_2^{p-n+1}, x_1^{n-1}x_2^{p-n}, x_1^{n-1}x_2^{p-n+1}\}\] if $n$ is odd, and 
\[\{1,x_2, x_1, x_1x_2, \ldots, x_1^{l-2}x_2, x_1^{l-1}, x_1^{l-1}x_2,\]\[ x_1^{n-l}x_2^{p-n}, x_1^{n-l}x_2^{p-n+1}, x_1^{n-l+1}x_2^{p-n} \ldots, x_1^{n-2}x_2^{p-n+1}, x_1^{n-1}x_2^{p-n}, x_1^{n-1}x_2^{p-n+1}\}\] if $n$ is even.

In either case, we note that none of the claimed generators have lead term divisible by $x_3$, that each has $x_1$-degree $<p$, that there are at most two elements in $S_n$ with the same $x_1$-degree, and that when this happens these elements have $x_2$-degrees differing by 1. Combined with Lemma \ref{obs}, we see that for every possible choice of $f \in A$ and $g \in S_n$, the lead monomial of $fg$ is different. Therefore there cannot be any $A$-linear relations between the elements of $S_n$.   
\end{proof}

\subsection{$V_2 \oplus V_2$}

In this subsection let $V \cong V_2 \oplus V_2$. Choose a basis $v_1,v_2,v_3,v_4$ of $V$ so that the action on the corresponding dual basis $x_1,x_2,x_3,x_4$ is given by
\begin{align*}
\sigma  x_1 &= x_1+x_3;\\
\sigma  x_2 &= x_2+x_4;\\
\sigma  x_3 &= x_3;\\
\sigma  x_4 &= x_4.
\end{align*}

It is easy to see that $\{N^G(x_1) = x_1^p-x_1x_3^{p-1}, N^G(x_2) = x_2^p-x_2x_4^{p-1},x_3,x_4\}$ is a homogeneous system of parameters. We also see easily that $u:=x_1x_4-x_2x_3$ is invariant, and does not belong to the algebra $A$ spanned by the given system of parameters. Campbell and Hughes \cite[Proposition~2.1]{CampbellHughes} showed that  $\{1,u, \ldots, u^{p-1}\}$ generate $\kk[V]^G$ as a free $A$-module. Therefore
\begin{equation}\label{sv3} s(\kk[V]^G,A) = \sum_{i=0}^{p-1}2i = p(p-1).\end{equation}

Noting that $G$ contains no pseudo-reflections, we obtain the following:
\begin{Lemma} We have $r(K_n,A) = np$ and $s(K_n,A) = np(p-1)$.
\end{Lemma}
\begin{proof} This follows from Equations \eqref{rv3} and \eqref{sv3} above and Proposition \ref{ssubalg}.
\end{proof}

We use a graded lexicographical order on monomials with $x_1>x_2>x_3>x_4$. An easy inductive argument establishes the following:

\begin{Lemma}\label{v2v2leadterms}
The lead term of $\Delta^k(x_1^ix_2^j)$ is:
\[\left\{ \begin{array}{lr} \frac{j!}{(j-k)!}x_1^i x_2^{j-k} x_4^k & k \leq j \\ \frac{i! j!}{(i-j-k)!}x_1^{i-k}x_3^{k-j}x_4^j & j<k \leq i+j. \end{array} \right.\]
\end{Lemma}

\begin{Corollary} Let $0 \leq i,j<p$. The weight of $x_1^ix_2^jx_3^kx_4^k$ is $\min(i+j+1,p)$.
\end{Corollary}

We now define, for $k \geq 0$,
\[M_k:= \{x_1^ix_2^j: i+j=k, i<p, j<p\}.\]
We use the notation
\[fM_k :=\{fm: m \in M_k\}\]
for any $f \in \kk[V]$ and
\[\Delta^rM_k:= \{\Delta^r(m): m \in M_k\} \] for any $r \geq 0$.

We are now able to state the main result of this subsection:

\begin{Theorem} Let $n \leq p$. The Kernel $K_n$ of $\Delta^n: \kk[V] \rightarrow \kk[V]$ is a free $A$-module generated by the set $S:= M \cup U \cup D$ where
\begin{align*}
M &= \bigcup_{k=0}^{n-1} M_k;\\
U &= \bigcup_{k=1}^{p-n} u^kM_{n-1};\\
D &= \bigcup_{k=2p-n}^{2p-2} \Delta^{p-n}M_k. 
\end{align*}
\end{Theorem}

\begin{proof}
The $A$-isomorphism $K_n \cong \kk[V,W]^G$ shows that $K_n$ is free, and $r(K_n,A) = np$, and $s(K_n,A) = np(p-1)$. 
The elements of $M_k$ for $k \leq n-1$ have weight $k+1$, so the weight of each element of $M$ is at most $n$. Since $u \in \kk[V]^G$, the weight of each element of $U$ is $n$. If $k \geq 2p-n$ then $k \geq p$, so in that case the weight of each element of $M_k$ is $p$. It follows that the weight of each element of $D$ is $n$. The shows that the proposed generators all lie in $K_n$. Now by Proposition \ref{freetest} it is enough to show that the given set is $A$-independent, has $np$ elements, and degree sum $np(p-1)$.

Viewed as a polynomial in $x_1,x_2$ with coefficients in $\kk[x_3,x_4]$, the monomials in $M_k$ each have total degree $k$. Therefore the elements of $M$ have total degree at most $n-1$. Further, each monomial in $M_k$ has a distinct $x_1,x_2$ bidegree. So the elements of $M$ have distinct bidegrees.

Meanwhile, the elements of $u^kM_{n-1}$ have total degree (in $x_1,x_2$) $k+n-1$. Therefore the total degrees of the elements of $U$ lie between $n$ and $p-1$ inclusive. Further, each element of $u^kM_{n-1}$ has distinct bidegree. So the bidegrees of the elements of $U$ are distinct from each other and also from all elements of $M$.

Let $m:= x_1^ix_2^j \in M_k$ with $k \geq 2p-n$. Then $p-n \leq k-p = i+j-p < j$, so by Lemma \ref{v2v2leadterms}, the lead term of $\Delta^{p-n}(m)$ is $$\frac{j!}{(j-p+n)!}x_1^ix_2^{j-p+n}x_4^{p-n}.$$
In particular, its total degree in $x_1,x_2$ is $i+j-p+n = k-p+n \geq p$. So the elements of $D$ have bidegree $\geq p$. Further, the bidegrees of the elements of $\Delta^{p-n}M_k$ are distinct from one another. We have shown that the bidegrees of the elements of $S$ are all distinct from one another.  

Multiplying an element of $\kk[V]$ by an element of $A$ preserves its $x_1$- and $x_2$- degree modulo $p$. Since each element of $S$ has $x_1$- and $x_2$- degree $<p$, and distinct $x_1,x_2$-bidegree, it follows that there is no $A$-linear relation between the elements of $S$. So $S$ is $A$-independent as required.

It remains to check the size and degree sum of $S$. The number of elements in $M_k$ is $k+1$ if $k < p$, and $2p-k-1$ otherwise. Thus
\begin{align*}
|S| &= |M|+|U|+|D|\\
&= \sum_{k=0}^{n-1} (k+1) + n(p-n) + \sum_{k=2p-n}^{2p-2} (2p-1-k)\\
&= \frac12 n(n+1) + n(p-n) + (2p-1)(n-1) - \frac12(n-1)(4p-n-2)\\
&= \frac12 n(n+1) + n(p-n) + \frac12 n(n-1)\\
&= \frac12 n \cdot (2n) + np-n^2\\
&=np
\end{align*}
as required. 

Finally, the degree of each element of $M_k$ is $k$, the degree of each element of $u^kM_{n-1}$ is $2k+n-1$ and the degree of each element of $\Delta^{p-n}M_k$ is also $k$. We note the identity
\[\sum_{k=0}^n k(k+1) = \frac13 n(n+1)(n+2).\]
Therefore the degree sum of $M$ is
\begin{align*}
\sum_{k=0}^{n-1} k(k+1) &= \frac13(n-1)n(n+1)\\
&= \frac13 (n^3-n).
\end{align*}
The degree sum of $U$ is
\begin{align*}
\sum_{k=1}^{p-n} (2k+n-1)n &= 2n \sum_{k=1}^{p-n}k + (p-n)(n-1)n\\
&= n(p-n)(p-n+1)+n(p-n)(n-1)\\
&=np^2-n^2p.
\end{align*}
And the degree sum of $D$ is
\begin{align*}
&\sum_{k=2p-n}^{2p-2} k(2p-1-k)\\
&= 2p \sum_{k=2p-n}^{2p-2}k - \sum_{k=2p-n}^{2p-2}k(k+1)\\
&= p(4p-2-n)(n-1) - \frac13(2p-2)(2p-1)(2p) + \frac13 (2p-n-1)(2p-n)(2p-n+1)\\
&= p(4p-2-n)(n-1) - \frac13(2p-2)(2p-1)(2p)\\
& \qquad + \frac13 (2p-2-(n-1))(2p-1-(n-1))(2p-(n-1))\\
&= p(4p-2-n)(n-1) - \frac13\left((6p-3)(n-1)^2 - (12p^2-12p+2)(n-1)-(n-1)^3 \right)\\
&= (n-1)\left(p(4p-2-n+\frac13(6p-3(n-1)-12p^2+12p-2-(n-1)^2)) \right)\\
&= (n-1)(pn-\frac13n - \frac13n^2 )\\
&= n^2p-np-\frac13(n^3-n).
\end{align*}
So the degree sum of $S$ is
\begin{align*}
&\frac13 (n^3-n)+np^2-n^2p=n^2p-np-\frac13(n^3-n)\\
& = np(p-1)
\end{align*}
as required. This completes the proof.
\end{proof}

\subsection{$V_3, p=2$}.
In this section we consider the representation $V =V_3$ of $C_4 = \langle \sigma \rangle$ over a field of characteristic two. Shank and Wehlau \cite{ShankWehlaupplusone} studied the modular invariant rings $\kk[V_{p+1}]^G$ where $G = C_{p^2}$, of which this is particularly simple example. They showed that the invariants were generated by the norms of the variables, and relative transfers from the unique proper subgroup of $G$. For our purposes we need not a fundamental generating set, but rather primary and secondary invariants. We proceed as follows:
Let $v_1,v_2,v_3$ be a basis of $V$ such that the action of $\sigma$ on the dual basis $\{x_1,x_2,x_3\}$is given by 
\begin{align*}\sigma x_1 &= x_1+x_2\\
\sigma x_2 &= x_2+x_3\\
\sigma x_3 &= x_3. \end{align*}
Once more we use a graded lexicographic order with $x_1>x_2>x_3$. We set $H = \langle \sigma^2 \rangle$. Note that the action of $H$ is given by

\begin{align*}\sigma^2 x_1 &= x_1+x_3\\
\sigma^2 x_2 &= x_2\\
\sigma^2 x_3 &= x_3. \end{align*}

It is easy to see that $\{N^G(x_1),N^G_H(x_2), x_3\}$ form a homogeneous system of parameters. Since $\dim(V^G)=1$, $\kk[V]^G$ is a Cohen-Macaulay ring, hence a free module over the subalgebra $A = \kk[N^G(x_1), N^G_H(x_2), x_3]$. Now by \cite[Theorem~3.71]{DerksenKemper} we have
\begin{equation}\label{rv3} r(\kk[V]^G,A) = \frac{8}{|G|} = 2 \end{equation} as the degree product of the generators of $A$ is $8$.

A direct computation establishes that $N^G_H(x_2)$ and $x_3^2$ are the only invariants of degree two. Further it is easily shown that  
$$u:= x_1^2x_3+x_1x_3^2+x_2^3+x_2^2x_3 \in \kk[V]^G,$$
and a lead term argument shows that $u \not \in A$. It must follow therefore that $\{1,u\}$ generates $\kk[V]^G$ freely over $A$.
In particular $s(\kk[V]^G,A) =3$.

Now by Proposition \ref{ssubalg}, we get 
\[r(K_n,A) = 2n,\]
and
\[s(K_n,A) = 3n+2 s(\kk[V,V_n]^H, \kk[V]^H),\]
the latter because $H$ is the subgroup of $G$ generated by pseudo-reflections.

We consider each value of $n=2,3,4$ separately. We note that $\Delta^2 = \sigma^2-\iota$, so $K_2 = \kk[V]^H = \kk[N^H(x_1),x_2,x_3]$. However, we seek generators of $K_2$ as an $A$-module. Our first result is:

\begin{prop}
The set $S = \{1,x_2,N^H(x_1),x_2N^H(x_1)\}$ generates $K_2$ over $A$.
\end{prop}

\begin{proof}
We will need to compute $s(\kk[V,V_2]^H, \kk[V]^H)$. Observe that, since $H$ acts trivially on $V_2$,  we may write
\[\kk[V,V_2]^H = \kk[V]^H \otimes V_2.\]
If $\{w_1,w_2\}$ is a basis of $V_2$, then $\kk[V,V_2]^H$ is generated over $\kk[V]^H$ by $w_1$ and $w_2$, both having degree zero. So $s(\kk[V,V_2]^H, \kk[V]^H) = 0$ and therefore $s(K_2,A) = 6$.

Now it is clear that every element of $S$ lies in $K_2 = \kk[V]^H$. Further, $|S|=4$ and the degree sum of $S$ is 6. It remains to show that $S$ is $A$-independent. This is easy to see: the lead terms of $S$ are $1,x_2,x_1^2$ and $x_1^2x_2$ and every element of $A$ has lead term $x_1^{4a}x_2^{2b}x_3$, so the lead terms of any different pair of elements of the form $fg$ where $f \in A$, $g \in S$ are distinct.
\end{proof}

\begin{prop}
The set $S = \{1,x_1,x_2,x_1^2,\Delta(x_1^3), \Delta(x_1^3x_2).\}$ generates $K_3$ over $A$.
\end{prop}

\begin{proof} 
We will need to compute $s(\kk[V,V_3]^H, \kk[V]^H)$. Choose a basis $\{w_1,w_2,w_3\}$ of (the second copy of) $V_3$ such that\begin{align*}\sigma^2 w_1 &= w_1+w_3\\
\sigma^2 w_2 &= w_2\\
\sigma^2 w_3 &= w_3. \end{align*}
Then we may write
\[\kk[V,V_3]^H = (S(V^*) \otimes V_3)^H = w_2 \otimes \kk[V]^H \oplus (S(V^*) \otimes W)^H\] 
where $W = \langle w_1,w_3 \rangle$.
The first direct summand is generated by $w_2$ over $\kk[V]^H$, with degree zero. The second is isomorphic to $\ker((\sigma^2-1)^2) = \kk[V]$, which is generated over $\kk[V]^H = \kk[N^H(x_1),x_2,x_3]$ by $x_1$. Therefore $$s(\kk[V,V_3]^H, \kk[V]^H) = 1.$$ Consequently we get that $s(\kk[V,V_3]^G, \kk[V]^G) = 11$.

Now an easy direct calculation shows that every element of $S$ lies in $K_3$ (for the last two simply note that $\Delta^4 = 0$). We have $|S| = 6$ and the degree sum of $S$ is 11. Further, the lead terms of $S$ are
\[\{1,x_1,x_2,x_1^2,x_1^2x_2,x_1^2,x_1^3x_3\}\] and so once more every for different pair of elements of the form $fg$ where $f \in A$ and $g \in S$, the lead terms are distinct. 
\end{proof}

\begin{prop}
The set $S = \{1,x_1,x_2,x_1^2,x_1x_2,x_1^3,x_1^2x_2,x_1^3x_2\}$ generates $K_4$ over $A$.
\end{prop}

\begin{proof}
We have $K_4 = \kk[V]$ so obviously each element of $S$ lies in $K_4$, and the proof of $A$-independence is similar to the previous result. To prove $S$ is a generating set we could proceed as before, where we would find that  $s(\kk[V,V_4]^H, \kk[V]^H) = 2$ and hence  $s(\kk[V,V_4]^G, \kk[V]^G) = 16$. However it is probably simpler to argue as follows: let $f \in \kk[V]$ and assume $f \not \in SA$ with minimal degree in $x_1$. By long division with $N^G(x_1)$ we may assume the $x_1$-degree of $f$ is at most three. Thus we may write
\[f = g_0+g_1x_1+g_2x_1^2+g_3x_1^3,\]
with $g_i \in \kk[x_2,x_3]$. Now assume that among all such examples, $g_3$ has minimal $x_2$-degree; by long division with $N^G_H(x_2)$ we may assume this $x_2$-degree is at most one. Treating each $g_i = 2,1,0$ in turn in the same manner, we see that $f \in SA$ after all, a contradiction.
\end{proof}

\bibliographystyle{plain}
\bibliography{MyBib}

\end{document}